\documentclass{article}
\usepackage[dvipsnames]{xcolor}
\usepackage{gastex}
\usepackage{amsmath}
\usepackage{amssymb}
\usepackage{upref}
\usepackage{multirow}
\usepackage{multicol}
\usepackage{graphicx}
\usepackage{url}
\usepackage{theorem}
\newtheorem{theorem}{Theorem}

\newtheorem{proposition}[theorem]{Proposition}

{\theorembodyfont{\rmfamily}%
  \newtheorem{example}[theorem]{Example}
   }
\newenvironment{proof}{\noindent\textit{Proof.}}
{\QED\vskip\theorempostskipamount} 

\def\petitcarre{\vrule height4pt width 4pt depth0pt}
\def\QED{\relax\ifmmode\eqno{\hbox{\petitcarre}}\else{%
  \unskip\nobreak\hfil\penalty50\hskip2em\hbox{}\nobreak\hfil
  \petitcarre
  \parfillskip=0pt \finalhyphendemerits=0\par\smallskip}
  \fi}
\newcommand\RR{\mathcal{R}}

\newcommand{\N}{\mathbb{N}}
\newcommand{\Z}{\mathbb{Z}}
\newcommand{\R}{\mathbb{R}}
\newcommand{\GT}{\mathbb{T}}

\def\un(#1){\underline{#1}\,}
\DeclareMathOperator{\Card}{Card}

\DeclareMathOperator{\Sep}{Sep}

\definecolor{ivoire}{rgb}{0.99,0.99,0.8}
\usepackage{calc}
\newcounter{hours}\newcounter{minutes}
\newcommand\computetime{\setcounter{hours}{\time/60}%
  \setcounter{minutes}{\time-\value{hours}*60}%
  \thehours\,h\,\theminutes}
\newcommand\dateandtime{\today\quad\computetime}
\numberwithin{theorem}{section}
\numberwithin{equation}{section}
\numberwithin{figure}{section}
\numberwithin{table}{section}
\title{Bifix codes and interval exchanges}
\author{Val\'erie Berth\'e$^1$, Clelia De Felice$^2$, 
Francesco Dolce$^3$, Julien Leroy$^4$,\\
 Dominique Perrin$^3$,
Christophe  Reutenauer$^5$,
Giuseppina Rindone$^3$\\\\
$^1$CNRS, Universit\'e Paris 7,
$^2$Universit\`a degli Studi di Salerno,\\
$^3$Universit\'e Paris Est, LIGM,
$^4$Universit\'e du Luxembourg,\\
 $^5$Universit\'e du Qu\'ebec \`a Montr\'eal}
\date{\dateandtime}
\begin{document}
%========================
%========== Lists ==================================
\makeatletter
\def\@listI{%
  \leftmargin\leftmargini
  \setlength{\parsep}{0pt plus 1pt minus 1pt}
  \setlength{\topsep}{2pt plus 1pt minus 1pt}
  \setlength{\itemsep}{0pt}
}
\let\@listi\@listI
\@listi
\def\@listii {%
  \leftmargin\leftmarginii
  \labelwidth\leftmarginii
  \advance\labelwidth-\labelsep
  \setlength{\topsep}{0pt plus 1pt minus 1pt}
}
\def\@listiii{%
  \leftmargin\leftmarginiii
  \labelwidth\leftmarginiii
  \advance\labelwidth-\labelsep
  \setlength{\topsep}{0pt plus 1pt minus 1pt}
%              \topsep    0\p@ \@plus\p@\@minus\p@
  \setlength{\parsep}{0pt} 
  \setlength{\partopsep}{1pt plus 0pt minus 1pt}
}
\makeatother
%====================
\maketitle
%========================

\begin{abstract}
We investigate the relation between bifix codes and interval exchange
transformations. We prove that the class of natural codings of regular 
interval echange transformations is closed under maximal bifix decoding.
\end{abstract}
\tableofcontents
\section{Introduction}
This paper is part of a research initiated in~\cite{BerstelDeFelicePerrinReutenauerRindone2012} which studies the connections between the three 
subjects formed by symbolic dynamics, the theory of codes and
combinatorial group theory. The initial focus was placed
on the classical case of Sturmian systems and progressively extended
to more general cases.

The starting point of the present research is the observation that
the family of
Sturmian sets is not closed under decoding by
a maximal bifix code, even in the more simple case of
the code formed of all words of fixed length $n$.
Actually, the decoding of the Fibonacci word (which 
 corresponds to a rotation of
angle $\alpha=(3-\sqrt{5})/2$) by blocks
of length $n$ is an interval exchange transformation
corresponding to a rotation of angle $n\alpha$
coded on $n+1$ intervals.
This has lead us to consider the set of factors
of interval exchange transformations, called interval
exchange sets.
Interval exchange transformations 
were introduced by Oseledec~\cite{Oseledec1966}
following
an earlier idea of Arnold~\cite{Arnold1963}. These transformations form a generalization
of rotations of the circle.

The main result in this paper is that
the family of regular interval exchange sets is closed
under decoding by a maximal bifix code
(Theorem~\ref{corollaryInverseImage}).
This result invited us to try to extend to regular interval
exchange transformations the results relating bifix codes
and Sturmian words. This lead us to generalize
in~\cite{BertheDeFeliceDolceLeroyPerrinReutenauerRindone2014}  
to a large class of sets 
the main result of~\cite{BerstelDeFelicePerrinReutenauerRindone2012},
namely 
the Finite Index Basis Theorem relating maximal bifix codes and
bases of subgroups of finite index of the free group.

Theorem~\ref{corollaryInverseImage} reveals a close connection between
maximal bifix codes and interval exchange transformations. Indeed,
given an interval exchange transformation $T$ each maximal bifix
code $X$ defines a new interval exchange transformation $T_X$.
We show at the end of the paper, using the Finite Index Basis Theorem,
that this transformation is actually an interval exchange transformation
on a stack, as defined in~\cite{BoshernitzanCarroll1997}
(see also~\cite{Veech1975}).

The paper is organized as follows.

In Section~\ref{sectionIntervalExchange}, we recall some notions
concerning interval exchange transformations. We state the result
of Keane~\cite{Keane1975} which proves that regularity is
 a sufficient condition for the minimality
of such a transformation (Theorem~\ref{theoremKeane}).

We study in Section~\ref{sectionBifixCodes} the relation between
interval exchange transformations and bifix codes.
We prove that the transformation associated with a finite $S$-maximal
bifix code is an interval exchange transformation
(Proposition~\ref{propositionTransformation}). We also prove
a result concerning the regularity of
 this transformation (Theorem~\ref{theoremMinimal}).

 We discuss the relation
with bifix codes and we show that the class of regular
interval exchange sets is closed under decoding by a maximal bifix
code, that is, under inverse images
by coding morphisms of finite
maximal bifix codes (Theorem~\ref{corollaryInverseImage}).

In Section~\ref{sectionTreePlanarTree} we introduce tree sets
and planar tree sets. We show, reformulating
 a theorem of \cite{FerencziZamboni2008},
that uniformly recurrent planar tree sets are the regular interval
exchange sets (Theorem~\ref{theoremFZ}). We show in another
paper~\cite{BertheDeFeliceDolceLeroyPerrinReutenauerRindone2013c}
that, in the same way as regular interval exchange sets,
 the class of uniformly recurrent tree sets is closed under
maximal bifix decoding.

In Section~\ref{sectionExchangePieces}, we explore a new direction,
extending  the results
of this paper to a more general case. We introduce exchange of pieces,
a notable example being given by the Rauzy fractal. We
indicate how the decoding of the natural codings of exchange of pieces by maximal bifix
codes are again natural codings of exchange of pieces. We finally give in Section~\ref{sectionSkew}
an alternative
proof of Theorem~\ref{corollaryInverseImage} using a skew product
of a regular interval exchange transformation with a finite permutation group.

\paragraph{Acknowledgements} This work was supported by grants from
R\'egion Ile-de-France, the ANR projects Eqinocs
and Dyna3S, the Labex Bezout,
the FARB Project
``Aspetti algebrici e computazionali nella teoria dei codici,
degli automi e dei linguaggi formali'' (University of Salerno, 2013)
and the MIUR PRIN 2010-2011 grant
``Automata and Formal Languages: Mathematical and Applicative Aspects''.
 We thank the referee
for his useful remarks on the first version of the paper which was initially
part of the companion 
paper~\cite{BertheDeFeliceDolceLeroyPerrinReutenauerRindone2014}.
%%%%%%%%%%%%%%%%%%%%%%%%%%%%%%%%%%
%\section{Interval exchange sets}\label{sectionIntervalExchange}
%%%%%%%%%%%%%%%%

\section{Interval exchange transformations}\label{sectionIntervalExchange}
Let us recall the definition of an interval exchange transformation
(see~\cite{CornfeldFominSinai1982} or~\cite{BertheRigo2010}).

A \emph{semi-interval} is a nonempty subset of the real line of the
form $[\alpha,\beta[=\{z\in\R\mid \alpha\le z<\beta\}$. Thus it
is a left-closed and right-open interval. For two semi-intervals
$\Delta,\Gamma$, we denote $\Delta<\Gamma$ if $x<y$
for any $x\in\Delta$ and $y\in\Gamma$. 

Let $(A,<)$ be an ordered set. A partition $(I_a)_{a\in A}$
of $[0,1[$ in semi-intervals is \emph{ordered}  if $a<b$
implies $I_a<I_b$.

Let $A$ be a finite set ordered by two total orders $<_1$ and $<_2$. Let 
$(I_a)_{a\in A}$ be  a partition of $[0,1[$ in semi-intervals ordered
  for $<_1$.
Let $\lambda_a$ be the length of $I_a$.
Let $\mu_a=\sum_{b\le_1 a}\lambda_b$ and $\nu_a=\sum_{b\le_2
  a}\lambda_b$.
Set $\alpha_a=\nu_a-\mu_a$.
The \emph{interval exchange transformation} relative to 
$(I_a)_{a\in A}$ is the  map $T:[0,1[\rightarrow [0,1[$ defined by
\begin{displaymath}
 T(z)=z+\alpha_a\quad \text{ if } z\in I_a.
\end{displaymath}
Observe that the restriction of $T$ to $I_a$ is a translation
onto $J_a=T(I_a)$, that $\mu_a$ is the right boundary of $I_a$
and that $\nu_a$ is the right boundary of $J_a$.
We additionally denote
 by $\gamma_a$ the left boundary of $I_a$ and
by $\delta_a$ the left boundary of $J_a$. Thus
\begin{displaymath}
I_a=[\gamma_a,\mu_a[,\quad J_a=[\delta_a,\nu_a[.
\end{displaymath}

Note that $a<_2 b$ implies $\nu_a<\nu_b$ and thus $J_a< J_b$. This
shows that the family
$(J_a)_{a\in A}$ is a partition of $[0,1[$ ordered for $<_2$.
In particular,
the transformation $T$ defines a bijection from $[0,1[$ onto itself.

An interval exchange transformation relative to 
$(I_a)_{a\in A}$ is also said to be on the alphabet $A$.
The values $(\alpha_a)_{a\in A}$ are called the \emph{translation
  values}
of the transformation $T$.
\begin{example}\label{exampleRotation}
 Let $R$ be the interval
exchange
transformation corresponding to $A=\{a,b\}$, $a<_1b$, $b<_2a$,
$I_a=[0,1-\alpha[$,
$I_b=[1-\alpha,1[$.
The transformation $R$ is the rotation  of angle $\alpha$ on the semi-interval $[0,1[$ 
defined by $R(z)=z+\alpha\bmod 1$.
\end{example}
Since $<_1$ and $<_2$ are total
orders, there exists a unique permutation $\pi$ of $A$
such that $a<_1b$ if and only if $\pi(a)<_2\pi(b)$.
Conversely, $<_2$ is determined by $<_1$ and $\pi$
and $<_1$ is determined by $<_2$ and $\pi$.
The permutation $\pi$ is said to be \emph{associated} with $T$.

If we set $A=\{a_1,a_2,\ldots,a_s\}$ with $a_1<_1a_2<_1\cdots<_1a_s$,
  the pair $(\lambda,\pi)$ formed by the family
$\lambda=(\lambda_a)_{a\in A}$ and the permutation
$\pi$
determines the map $T$. We will also denote $T$ as $T_{\lambda,\pi}$.
The transformation $T$ is also said to be an $s$-interval exchange
transformation.

It is easy to verify that if $T$ is an interval exchange
transformation,
then $T^n$ is also an interval exchange transformation for
any $n\in\Z$.

\begin{example}
A $3$-interval exchange transformation is represented in
Figure~\ref{figure3interval1}. One has $A=\{a,b,c\}$ with
$a<_1b<_1c$ and $b<_2c<_2a$.
The associated permutation is the cycle $\pi=(abc)$.
%Let $\alpha,\beta,\gamma$ be positive real numbers such that
%$\alpha+\beta+\gamma=1$.
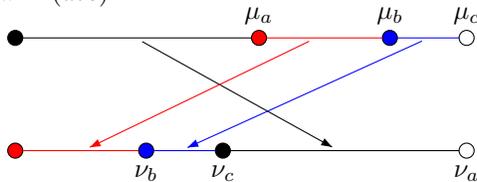
\begin{figure}[hbt]
\centering
\gasset{Nadjust=wh,AHnb=0}
\begin{picture}(60,20)(0,-5)

\node[fillcolor=black](0H)(0,15){}
\node[Nframe=n](1HH)(32.4,18){$\mu_a$}\node[fillcolor=red](1H)(32.4,15){}
\node[Nframe=n](2HH)(49.8,18){$\mu_b$}\node[fillcolor=blue](2H)(49.8,15){}
\node[Nframe=n](3HH)(60,18){$\mu_c$}\node(3H)(60,15){}
\node[Nframe=n](05H)(16,15){}\node[Nframe=n](15H)(40,15){}\node[Nframe=n](25H)(55,15){}
\node[fillcolor=red](0B)(0,0){}
\node[Nframe=n](1BB)(17.4,-3){$\nu_b$}\node[fillcolor=blue](1B)(17.4,0){}
\node[Nframe=n](2BB)(27.6,-3){$\nu_c$}\node[fillcolor=black](2B)(27.6,0){}
\node[Nframe=n](3BB)(60,-3){$\nu_a$}\node(3B)(60,0){}
\node[Nframe=n](05B)(9,0){}\node[Nframe=n](15B)(22,0){}\node[Nframe=n](25B)(43,0){}

\drawedge(0H,1H){}
\drawedge[linecolor=red](1H,2H){}
\drawedge[linecolor=blue](2H,3H){}
\drawedge[linecolor=red](0B,1B){}\drawedge[linecolor=blue](1B,2B){}\drawedge(2B,3B){}
\drawedge[AHnb=1](05H,25B){}\drawedge[AHnb=1,linecolor=red](15H,05B){}
\drawedge[AHnb=1,linecolor=blue](25H,15B){}
\end{picture}
\caption{A $3$-interval exchange transformation.}\label{figure3interval1}
\end{figure}
\end{example}
%%%%%%%%%%%%%%%%%%%%%%%%%%%%%%%%%
\subsection{Regular interval exchange transformations}

The \emph{orbit} of a point $z\in[0,1[$ is the set $\{T^n(z)\mid
  n\in\Z \}$. 
The transformation $T$ is said to be \emph{minimal} if, for any
  $z\in[0,1[$,
the  orbit of $z$ is dense in $[0,1[$.

Set $A=\{a_1,a_2,\ldots,a_s\}$ with $a_1<_1 a_2<_1\ldots<_1 a_s$,
 $\mu_i=\mu_{a_i}$ and $\delta_i=\delta_{a_i}$. The points $0,\mu_1,\ldots,\mu_{s-1}$
form the set of \emph{separation points} of $T$, denoted $\Sep(T)$.
Note that the singular points of the transformation $T$ (that is
the points $z\in[0,1[$ at which $T$ is not continuous) are
among the separation points but that the converse is not true in
general (see Example~\ref{exampleBifixDegree2}).

An interval exchange transformation $T_{\lambda,\pi}$ is called
\emph{regular} if the orbits of the nonzero separation
points $\mu_1,\ldots,\mu_{s-1}$ are infinite and disjoint.
Note that the orbit of $0$ cannot be disjoint of the others
since one has $T(\mu_i)=0$ for some $i$ with $1\le i\le s-1$.
The term regular was introduced by Rauzy in~\cite{Rauzy1979}.
A regular interval exchange transformation is also said
to be \emph{without connections} or to
satisfy the \emph{idoc} condition (where idoc stands for infinite
disjoint orbit condition).

Note that  since $\delta_2=T(\mu_1),\ldots,\delta_s=T(\mu_{s-1})$,
$T$ is regular if and only if the orbits of $\delta_2,\ldots,\delta_s$
are infinite and disjoint.

As an example, the $2$-interval exchange transformation
of Example~\ref{exampleRotation}
which is the rotation of angle $\alpha$ is regular
if and only if $\alpha$ is irrational.

Note that if $T$ is a regular $s$-interval exchange transformation,
then for any $n\ge 1$, the transformation $T^n$ is an
$n(s-1)+1$-interval exchange transformation. Indeed,
the points $T^i(\mu_j)$ for $0\le i\le n-1$ and $1\le j\le s-1$
are distinct and define a partition in $n(s-1)+1$ intervals.

The following result is due to Keane~\cite{Keane1975}.
\begin{theorem}[Keane]\label{theoremKeane}
A regular interval exchange transformation is minimal.
\end{theorem}

The converse is not true. Indeed, consider the rotation of angle $\alpha$
with $\alpha$ irrational, as a $3$-interval exchange transformation
with $\lambda=(1-2\alpha,\alpha,\alpha)$ and
$\pi=(132)$. The transformation is minimal as any rotation of
irrational angle but
it is not regular since $\mu_1=1-2\alpha$, $\mu_2=1-\alpha$ and
thus $\mu_2=T(\mu_1)$.

The following necessary condition for minimality
of an interval exchange transformation is useful.
A permutation $\pi$ of an ordered set $A$ is called
\emph{decomposable}
if there exists an element $b\in A$ such that the set $B$
of elements strictly less than $b$ is nonempty and such that
$\pi(B)=B$. Otherwise it is called
\emph{indecomposable}. If an interval exchange transformation
$T=T_{\lambda,\pi}$ is minimal, the permutation $\pi$ is
indecomposable.
Indeed, if $B$ is a set as above, the  set $S=\cup_{a\in B}I_a$  is 
closed under $T$ and strictly included in
$[0,1[$.

  The following example shows  that the indecomposability of $\pi$ is not 
sufficient for $T$ to be minimal.
\begin{example}
Let $A=\{a,b,c\}$ and $\lambda$ be such that $\lambda_a=\lambda_c$.
Let $\pi$ be the transposition $(ac)$. Then $\pi$ is indecomposable
but $T_{\lambda,\pi}$ is not minimal since it is the identity on $I_b$.
\end{example}
%%%%%%%%%%%%%%
\subsection{Natural coding}

Let $A$ be a finite nonempty alphabet. All words considered below,
unless
stated explicitly, are supposed to be on the alphabet $A$.
We denote by $A^*$ the set of all words on $A$.
We denote by $1$ or by $\varepsilon$ the empty word. We refer to
\cite{BerstelPerrinReutenauer2009} for the notions of prefix, suffix,
factor of a word.

 Let $T$ be an interval exchange transformation relative to 
$(I_a)_{a\in A}$. 
For a given real number $z\in[0,1[$,
the \emph{natural coding} of $T$ relative to $z$ is the infinite word
$\Sigma_T(z)=a_0a_1\cdots$ on the alphabet $A$
defined by 
\begin{displaymath}
a_n=a\quad \text{ if }\quad T^n(z)\in I_{a}.
\end{displaymath}

For a word $w=b_0b_1\cdots b_{m-1}$, 
let $I_w$ be the set
\begin{equation}
I_w=I_{b_0}\cap T^{-1}(I_{b_1})\cap\ldots\cap T^{-m+1}(I_{b_{m-1}}).\label{eqIu}
\end{equation}
Note that each $I_w$ is a semi-interval. Indeed, this is true if $w$ is
a letter. Next,  assume that $I_w$ is a semi-interval. Then
for any $a\in A$,
$T(I_{aw})=T(I_a)\cap I_w$ is a semi-interval since $T(I_a)$ is a
semi-interval
by definition of an interval exchange transformation. Since $I_{aw}\subset
I_a$, $T(I_{aw})$ is a translate of $I_{aw}$, which is therefore also
a semi-interval. This proves the property by induction on the length.

Set $J_w=T^m(I_w)$. Thus
\begin{equation}
J_w= T^{m}(I_{b_0})\cap T^{m-1}(I_{b_1})\cap\ldots\cap T(I_{b_{m-1}}).\label{eqJu}
\end{equation}
In particular, we have $J_a=T(I_a)$ for $a\in A$.
Note that each $J_w$ is a semi-interval. Indeed, this is true
if $w$ is a letter. Next, for any $a\in A$, we have 
$T^{-1}(J_{wa})=J_w\cap I_a$. This implies as above that $J_{wa}$ is
a semi-interval and proves the property by induction.
We set by convention $I_\varepsilon=J_\varepsilon=[0,1[$.
Then one has for any $n\ge 0$
\begin{equation}
a_na_{n+1}\cdots a_{n+m-1}=w \Longleftrightarrow T^n(z)\in I_w\label{eqIw}
\end{equation}
and
\begin{equation}
a_{n-m}a_{n-m+1}\cdots a_{n-1}=w \Longleftrightarrow T^n(z)\in J_w.\label{eqJw}
\end{equation}
Let $(\alpha_a)_{a\in A}$ be the translation values of $T$. Note that 
for any word $w$,
\begin{equation}
J_w=I_w+\alpha_w \label{eqJw2} 
\end{equation}
with  $\alpha_w=\sum_{j=0}^{m-1}\alpha_{b_j}$ as one may verify by
induction on $|w|=m$. Indeed it is true for $m=1$. For $m\ge 2$, 
set $w=ua$ with $a=b_{m-1}$. One has
$T^m(I_w)=T^{m-1}(I_w)+\alpha_{a}$
and $T^{m-1}(I_w)=I_w+\alpha_u$ by the induction
hypothesis
and the fact that $I_w$ is included in $I_u$. Thus
$J_w=T^m(I_w)=I_w+\alpha_u+\alpha_a=I_w+\alpha_w$.
Equation~\eqref{eqJw2} shows in particular that the restriction
of $T^{|w|}$ to $I_w$ is a translation.

%%%%%%%%%%%%%%%%%%%%%%%%%%%
\subsection{Uniformly recurrent sets}

A set $S$ of words on the alphabet $A$
is said to be \emph{factorial} if it contains the
factors of its elements. 

A factorial set is said to be \emph{right-extendable}
if for every $w\in S$ there is some $a\in A$ such that $wa\in S$.
It is \emph{biextendable} if for any $w\in S$, there are $a,b\in A$
such that $awb\in S$.

A set of words $S\ne \{\varepsilon\}$ is \emph{recurrent} if it is factorial and if for every
$u,w\in S$ there is a $v\in S$ such that $uvw\in S$. A recurrent set
 is biextendable.
It is said to be \emph{uniformly recurrent} if it is 
right-extendable and if, for any word $u\in S$, there exists an integer $n\ge
1$
such that $u$ is a factor of every word of $S$ of length $n$.
A uniformly recurrent set is recurrent.

We denote by $A^{\N}$ the set of infinite words on the alphabet $A$.
For a set $X\subset A^\N$, we denote by $F(X)$ the set of factors
of the words of $X$. 

Let $S$ be a set of words on the alphabet $A$.
For $w\in S$, set $R(w)=\{a\in A\mid wa\in S\}$ and
$L(w)=\{a\in A\mid aw\in S\}$.
A word $w$ is called \emph{right-special} if $\Card(R(w))\ge 2$
and  \emph{left-special} if $\Card(L(w))\ge 2$.
It is \emph{bispecial} if it is both right and left-special.

An infinite word on a binary alphabet is \emph{Sturmian} if its set
of factors is closed under reversal and if for each $n$
there is exactly one right-special word of length $n$.

An infinite word is a \emph{strict episturmian} word if its set
of factors is closed under reversal and for each $n$ there is
 exactly one right-special word $w$ of length $n$, which is
moreover such that $\Card(R(w))=\Card(A)$.

A morphism $f:A^*\rightarrow A^*$ is called \emph{primitive} if there
is an integer $k$ such that for all $a,b\in A$, the letter $b$
appears in $f^k(a)$. If $f$ is a primitive morphism, the set
of factors of any fixpoint
of $f$ is uniformly recurrent (see~\cite[Proposition 1.2.3]{PytheasFogg2002},
 for example).

\begin{example}\label{exampleFibonacci}
Let $A=\{a,b\}$.
The Fibonacci word is the fixpoint $x=f^\omega(a)=abaababa\ldots$ of the
morphism $f:A^*\rightarrow A^*$ defined by $f(a)=ab$ and $f(b)=a$.
It is a Sturmian word (see~\cite{Lothaire2002}). The set $F(x)$ of factors of $x$ is the
\emph{Fibonacci set}.
\end{example}
\begin{example}\label{exampleTribonacci}
Let $A=\{a,b,c\}$.
The Tribonacci word is the fixpoint $x=f^\omega(a)=abacaba\cdots$ of the morphism
$f:A^*\rightarrow A^*$ defined by $f(a)=ab$, $f(b)=ac$, $f(c)=a$.
It is a strict  episturmian word (see~\cite{JustinVuillon2000}).
The set $F(x)$ of factors of $x$ is the \emph{Tribonacci set}.
\end{example} 
%%%%%%%%%%%%%%%%%
\subsection{Interval exchange sets}

Let $T$ be an interval exchange set. The set $F(\Sigma_T(z))$ is called
an \emph{interval exchange set}. It is biextendable.

If $T$ is a minimal interval exchange transformation, one has $w\in F(\Sigma_T(z))$ if
and only
if $I_w\ne\emptyset$.
Thus the set $F(\Sigma_T(z))$ does not depend on $z$.
 Since it depends only on $T$, we denote
it by $F(T)$. When $T$ is regular (resp. minimal), such a set is called a
\emph{regular interval exchange set} (resp. a minimal interval exchange set).

Let $T$ be an interval exchange transformation.
Let $M$ be the closure in $A^\N$ of the set of all $\Sigma_T(z)$ for $z\in [0,1[$ and let $\sigma$
be the shift on $M$. The pair $(M,\sigma)$ is a \emph{symbolic dynamical system}, formed
of a topological space $M$ and a continuous transformation $\sigma$.
Such a system is said to be \emph{minimal} if the only closed
subsets invariant by $\sigma$ are $\emptyset$ or $M$
(that is, every orbit is dense). It is well-known
that $(M,\sigma)$ is minimal if and only if $F(T)$ is uniformly recurrent
(see for example~\cite[Theorem 1.5.9]{Lothaire2002} ).

We have the following commutative diagram (Figure~\ref{figureDiagram}).
\begin{figure}[hbt]
\gasset{Nframe=n}\centering
\begin{picture}(20,20)
\node(T1)(0,20){$[0,1[$}\node(T2)(20,20){$[0,1[$}
\node(S1)(0,0){$M$}\node(S2)(20,0){$M$}

\drawedge(T1,T2){$T$}\drawedge(T1,S1){$\Sigma_T$}
\drawedge(S1,S2){$\sigma$}\drawedge(T2,S2){$\Sigma_T$}
\end{picture}
\caption{The transformations $T$ and $\sigma$.}\label{figureDiagram}
\end{figure}

The map $\Sigma_T$ is neither continuous nor surjective. This
can be corrected by embedding the interval $[0,1[$ into a larger
space on which $T$ is a homeomophism 
(see~\cite{Keane1975} or~\cite[page 349]{BertheRigo2010}). However,
if the transformation $T$ is minimal, the
symbolic dynamical system $(M,S)$ is minimal 
(see~\cite[page 392]{BertheRigo2010}). Thus, we obtain the following
statement.
\begin{proposition}\label{propositionRegularUR}
For any minimal interval exchange transformation $T$,
the set $F(T)$ is uniformly recurrent.
\end{proposition}

Note that for a minimal interval exchange transformation $T$, the
map $\Sigma_T$ is injective (see \cite{Keane1975} page 30).

The following is an elementary property of the intervals $I_u$
which will be used below. We denote by $<_1$ the lexicographic order on $A^*$
induced by the order $<_1$ on $A$.

\begin{proposition}\label{propIu}
 One has $I_u< I_v$ if and only if $u<_1v$ and $u$ is
not a prefix of $v$. 
\end{proposition}
\begin{proof}
For a word $u$ and a letter $a$, it results from~\eqref{eqIu}
that $I_{ua}=I_u\cap T^{-|u|}(I_a)$.
Since $(I_a)_{a\in A}$ is an ordered partition, this implies that
$(T^{|u|}(I_u)\cap I_a)_{a\in A}$ is an ordered partition of 
$T^{|u|}(I_u)$. Since the restriction of $T^{|u|}$ to $I_u$
is a translation, this implies that
$(I_{ua})_{a\in A}$ is an ordered partition of $I_u$. Moreover, for two words $u,v$, it
results
also from~\eqref{eqIu} that $I_{uv}=I_u\cap T^{-|u|}(I_v)$. Thus
$I_{uv}\subset I_{u}$.

Assume that $u<_1v$ and that $u$ is not a prefix of $v$. Then
$u=\ell a s$ and $v=\ell b t$ with $a,b$ two letters such that $a<_1b$.
Then we have $I_{\ell a}< I_{\ell b}$, with $I_u\subset I_{\ell a}$
and $I_v\subset I_{\ell b}$ whence $I_u<I_v$.

Conversely, assume that $I_u<I_v$.  Since $I_u\cap I_v=\emptyset$,
the words $u,v$ cannot be comparable for
the prefix order. Set $u=\ell a s$ and $v=\ell b t$ with $a,b$ two distinct
letters.
If $b<_1a$, then $I_v< I_u$ as we have shown above. Thus $a<_1b$ which
implies $u<_1v$.
\end{proof}

We denote by $<_2$ the order on $A^*$ defined by
$u<_2 v$ if $u$ is a proper suffix of $v$ or if
$u=waz$ and $v=tbz$ with $a<_2b$. Thus $<_2$ is the
lexicographic
order on the reversal of the words induced by the order $<_2$ on the alphabet.

We denote by $\pi$ the morphism from $A^*$ onto itself
which extends to $A^*$ the permutation $\pi$ on $A$.
Then $u<_2 v$ if and only if $\pi^{-1}(\tilde{u})<_1\pi^{-1}(\tilde{v})$,
where $\tilde{u}$ denotes the reversal of the word $u$.

The following statement is the analogue of Proposition~\ref{propIu}.
\begin{proposition}\label{propJu}
Let $T_{\lambda,\pi}$ be an interval exchange transformation.
  One has $J_u< J_v$ if and only if
$u<_2 v$  and $u$ is not a suffix of $v$.
\end{proposition}
\begin{proof}
Let $(I'_a)_{a\in A}$ be the family of semi-intervals defined by
$I'_a=J_{\pi(a)}$.
 Then the interval exchange transformation $T'$
relative to $(I'_a)$ with translation values $-\alpha_a$ is the inverse of the
transformation $T$. The semi-intervals $I'_w$ defined by
Equation~\eqref{eqIu}
with respect to $T'$ satisfy $I'_w=J_{\pi(\tilde{w})}$
or equivalently $J_w=I'_{\pi^{-1}(\tilde{w})}$. Thus,
$J_u< J_v$ if and only if $I'_{\pi^{-1}(\tilde{u})}< I'_{\pi^{-1}(\tilde{v})}$
if and only if (by Proposition~\ref{propIu}) $\pi^{-1}(\tilde{u})<_1\pi^{-1}(\tilde{v})$ or equivalently
$u<_2 v$.
\end{proof}

%%%%%%%%%%%%%%%%%%%%%%%%%%%%%%%%%
\section{Bifix codes and interval exchange}\label{sectionBifixCodes}
In this section, we first introduce prefix codes and bifix codes.
For a more detailed exposition, see~\cite{BerstelPerrinReutenauer2009}.
We describe the link between maximal bifix codes and interval
exchange transformations and we prove our main result (Theorem~\ref{corollaryInverseImage}).

\subsection{Prefix codes and bifix codes}
A \emph{prefix code} is a set of nonempty words which does not contain any
proper prefix of its elements. A suffix code is defined symmetrically.
A  \emph{bifix code} is a set which is both a prefix code and a suffix
code.

A \emph{coding morphism} for a prefix code $X\subset A^+$ is a morphism
$f:B^*\rightarrow A^*$ which maps bijectively $B$ onto $X$.

Let $S$ be a set of words. A prefix code $X\subset S$ is $S$-maximal 
if it is not properly contained in any prefix code 
$Y\subset S$. Note that if $X\subset S$ is an $S$-maximal prefix code, 
any word of $S$ is comparable for the prefix order with a word of $X$.

A map $\lambda:A^*\rightarrow [0,1]$ such that $\lambda(\varepsilon)=1$
and, for any word $w$
\begin{equation}
\sum_{a\in A}\lambda(aw)=\sum_{a\in A}\lambda(wa)=\lambda(w),\label{eqProba}
\end{equation}
is called an \emph{invariant probability distribution} on $A^*$.

Let $T_{\lambda,\pi}$ be an interval exchange transformation.
For any word $w\in A^*$, denote by  $|I_w|$ the length of the
semi-interval $I_w$ defined by Equation~\eqref{eqIu}. Set $\lambda(w)=|I_w|$.
Then $\lambda(\varepsilon)=1$
and for any word $w$, Equation~\eqref{eqProba} holds and
thus $\lambda$ is an invariant probability distribution.

The fact that $\lambda$ is an invariant probability measure is
equivalent to the fact that the Lebesgue measure on $[0,1[$
is invariant by $T$. It is known that almost all regular
interval exchange transformations have no other invariant
probability measure (and thus are uniquely ergodic,
see~\cite{BertheRigo2010} for references).

\begin{example}\label{exampleFibonacciAlpha}
Let $S$ be the set of factors of the Fibonacci word (see Example~\ref{exampleFibonacci}).
It is the natural coding of the rotation of angle
$\alpha=(3-\sqrt{5})/2$ with respect to $\alpha$
(see~\cite[Chapter 2]{Lothaire2002}).
The values of the map $\lambda$ on the words of length at most $4$
in $S$ are indicated in
Figure~\ref{figProbaFibo}.
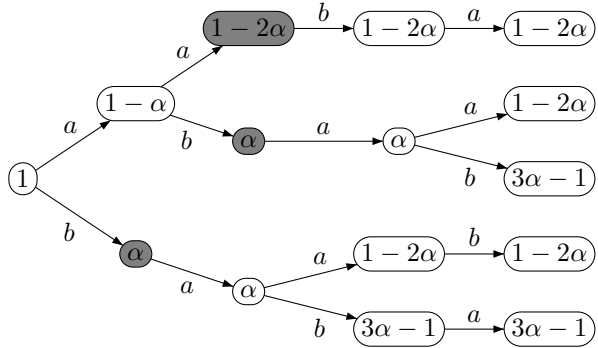
\begin{figure}[hbt]
\centering
\gasset{Nadjust=wh}
\begin{picture}(100,50)
\node(1)(0,25){$1$}
\node(a)(15,35){$1-\alpha$}\node[fillgray=.5](b)(15,15){$\alpha$}
\node[fillgray=.5](aa)(30,45){$1-2\alpha$}\node[fillgray=.5](ab)(30,30){$\alpha$}
\node(ba)(30,10){$\alpha$}
\node(aab)(50,45){$1-2\alpha$}
\node(aba)(50,30){$\alpha$}
\node(baa)(50,15){$1-2\alpha$}\node(bab)(50,5){$3\alpha-1$}
\node(aaba)(70,45){$1-2\alpha$}
\node(abaa)(70,35){$1-2\alpha$}\node(abab)(70,25){$3\alpha-1$}
\node(baab)(70,15){$1-2\alpha$}
\node(baba)(70,5){$3\alpha-1$}
\drawedge(1,a){$a$}\drawedge[ELside=r](1,b){$b$}
\drawedge(a,aa){$a$}\drawedge[ELside=r](a,ab){$b$}
\drawedge(aa,aab){$b$}\drawedge(ab,aba){$a$}
\drawedge(aab,aaba){$a$}
\drawedge(aba,abaa){$a$}\drawedge[ELside=r](aba,abab){$b$}
\drawedge[ELside=r](b,ba){$a$}
\drawedge(ba,baa){$a$}\drawedge[ELside=r](ba,bab){$b$}
\drawedge(baa,baab){$b$}\drawedge(bab,baba){$a$}
\end{picture}
\caption{The invariant probability distribution on the 
  Fibonacci set.}\label{figProbaFibo}
\end{figure}
\end{example}

The following result is a particular case of a result 
from~\cite{BerstelDeFelicePerrinReutenauerRindone2012} (Proposition 3.3.4).

\begin{proposition}\label{propositionProbMax}
Let $T$ be a minimal interval exchange transformation,
let $S=F(T)$ and let $\lambda$ be an invariant
probability
distribution on $S$. For any finite $S$-maximal prefix code $X$, one has
$\sum_{x\in X}\lambda(x)=1$.
\end{proposition}

The following statement is connected with Proposition~\ref{propositionProbMax}.
\begin{proposition}\label{propositionPrefixPartition}
Let $T$ be a minimal interval exchange transformation relative to 
$(I_a)_{a\in A}$, let $S=F(T)$ and
let $X$ be a finite $S$-maximal prefix code ordered by $<_1$.
 The family $(I_w)_{w\in X}$  is an ordered
partition of $[0,1[$.
\end{proposition}
\begin{proof}
By Proposition~\ref{propIu}, the sets $(I_w)$  for $w\in X$ are
pairwise disjoint. Let $\pi$ be the invariant probability
distribution
on $S$ defined by $\pi(w)=|I_w|$. By
Proposition~\ref{propositionProbMax},
we have $\sum_{w\in X}\pi(w)=1$. Thus the family $(I_w)_{w\in X}$
is a partition of $[0,1[$. By Proposition~\ref{propIu} it is an
  ordered partition.
\end{proof}

\begin{example}
Let $T$ be the rotation of angle $\alpha=(3-\sqrt{5})/2$. The set $S=F(T)$
is the Fibonacci set. The set $X=\{aa,ab,b\}$ is an $S$-maximal prefix
code (see the grey nodes in Figure~\ref{figProbaFibo}). The  partition of $[0,1[$ corresponding to $X$ is
\begin{displaymath}
I_{aa}=[0,1-2\alpha[,\quad I_{ab}=[1-2\alpha,1-\alpha[,\quad I_{b}=[1-\alpha,1[.
\end{displaymath}
The values of the lengths of the semi-intervals (the invariant
probability distribution) can also be read on
Figure~\ref{figProbaFibo}.

\end{example}
A symmetric statement holds for an $S$-maximal suffix code, namely
that the family $(J_w)_{w\in X}$  is an ordered partition of $[0,1[$
for the order $<_2$ on $X$.

\subsection{Maximal bifix codes}
Let $S$ be a set of words.
A bifix code $X\subset S$ is $S$-maximal if it is
not properly contained in a bifix code $Y\subset S$.
For a recurrent set $S$, a finite bifix code is $S$-maximal as a bifix code if
and only if it is an $S$-maximal prefix code 
(see~\cite[Theorem 4.2.2]{BerstelDeFelicePerrinReutenauerRindone2012}). 

A \emph{parse} of a word $w$ with respect to a bifix code $X$ is
a triple $(v,x,u)$ such that $w=vxu$ where $v$ has no suffix in $X$,
$u$ has no prefix in $X$ and $x\in X^*$.
We denote by $\delta_X(w)$ the number of parses of $w$ with respect to $X$.

The number of parses of a word $w$ is also equal to the number
of suffixes of $w$ which have no prefix in $X$ and the
 number of prefixes of $w$ which have no suffix in $X$
(see Proposition 6.1.6 in~\cite{BerstelPerrinReutenauer2009}).

By definition, the $S$-\emph{degree} of a bifix code
 $X$, denoted $d_X(S)$, is the maximal number
of parses of a word in $S$.  It can be finite or infinite.

The set of \emph{internal factors} of a set of words $X$,
denoted $I(X)$,
is the set of words $w$ such that 
there exist nonempty words $u,v$ with $uwv\in X$.

Let $S$ be a recurrent set and let
 $X$ be a finite $S$-maximal bifix code of $S$-degree $d$.
A word $w\in S$ is such that
$\delta_X(w)< d$ if and only if it is an internal factor of $X$,
that is,
\begin{displaymath}
I(X)=\{w\in S\mid \delta_X(w)<d\}
\end{displaymath}
(Theorem 4.2.8 in~\cite{BerstelDeFelicePerrinReutenauerRindone2012}).
Thus any word of $S$ which is not a factor of $X$ has $d$ parses.
This implies that
the $S$-degree $d$ is finite.
\begin{example}\label{exampleUniform}
Let $S$ be a recurrent set. For any integer $n\ge 1$, the set
$S\cap A^n$ is an $S$-maximal bifix code of $S$-degree $n$.
\end{example}
The \emph{kernel} of a bifix code $X$ is the set $K(X)=I(X)\cap X$.
Thus it is the set of words of $X$ which are also internal factors of
$X$.
By Theorem 4.3.11
of~\cite{BerstelDeFelicePerrinReutenauerRindone2012},
 a finite
$S$-maximal bifix code is determined by its $S$-degree and its kernel.

\begin{example}\label{exampleKernel}
Let $S$ be the Fibonacci set.  The set $X=\{a,baab,bab\}$ is the unique
$S$-maximal bifix code of $S$-degree $2$ with kernel $\{a\}$. Indeed,
the word $bab$ is not an internal factor and has two parses,
namely $(1,bab,1)$ and $(b,a,b)$.
\end{example}

The following result shows that bifix codes have a natural
connection with interval exchange transformations.

\begin{proposition}\label{propBifixPartition}
If $X$ is a finite $S$-maximal bifix code,
with $S$ as in Proposition~\ref{propositionPrefixPartition}, the families $(I_w)_{w\in X}$ and $(J_w)_{w\in X}$
are ordered partitions of $[0,1[$, relatively to the orders $<_1$ and
$<_2$ respectively.
\end{proposition}
\begin{proof}
This results from Proposition~\ref{propositionPrefixPartition}
and its symmetric
and from the fact that, since $S$ is recurrent, a finite $S$-maximal bifix code
is both an $S$-maximal prefix code and an $S$-maximal suffix code.
\end{proof}
Let $T$ be a regular interval exchange transformation
relative to  $(I_a)_{a\in A}$. Let $(\alpha_a)_{a\in A}$ be the
translation values of $T$.
Set $S=F(T)$.
Let $X$ be a finite $S$-maximal bifix code on the alphabet $A$.

Let $T_X$ be the transformation on $[0,1[$ defined by
\begin{displaymath}
T_X(z)=T^{|u|}(z)\quad\text{if}\quad z\in I_{u}
\end{displaymath}
with $u\in X$. The transformation is well-defined since, by
Proposition
\ref{propBifixPartition}, the family $(I_u)_{u\in X}$ is a partition of $[0,1[$.

Let $f:B^*\rightarrow A^*$ be a coding morphism for $X$.
Let $(K_b)_{b\in B}$ be the family of semi-intervals indexed by the alphabet
$B$
with $K_b=I_{f(b)}$. We consider $B$ as ordered
by the orders $<_1$ and $<_2$ induced by $f$.
Let $T_f$ be the interval exchange transformation
relative to  $(K_b)_{b\in B}$. Its translation
values are $\beta_b=\sum_{j=0}^{m-1}\alpha_{a_j}$
for $f(b)=a_0a_1\cdots a_{m-1}$. The transformation $T_f$ is called the 
\emph{transformation associated} with $f$. 

\begin{proposition}\label{propositionTransformation}
Let $T$ be a regular interval exchange transformation relative
to $(I_a)_{a\in A}$ and let $S=F(T)$. If $f:B^*\rightarrow A^*$ is a coding
morphism for a finite $S$-maximal bifix code $X$,
one has $T_f=T_X$.
\end{proposition}
\begin{proof} By 
Proposition~\ref{propBifixPartition}, the family  $(K_b)_{b\in B}$
is a partition of $[0,1[$ ordered by $<_1$.
For any $w\in X$, we have by Equation~\eqref{eqJw2}
$J_w=I_w+\alpha_w$
and thus $T_X$ is the interval exchange transformation relative to
$(K_b)_{b\in B}$ with translation values $\beta_b$.
\end{proof}
In the sequel, 
under the hypotheses of Proposition~\ref{propositionTransformation},
we consider $T_f$ as an interval exchange transformation. In
particular,
the natural coding of $T_f$ relative to $z\in [0,1[$ 
is well-defined.

\begin{example}\label{exampleBifixDegree2}
Let $S$ be the Fibonacci set. It is the set of factors of the
Fibonacci word, which is a natural coding of the rotation of
angle $\alpha=(3-\sqrt{5})/2$ relative to $\alpha$
(see Example~\ref{exampleFibonacciAlpha}).
Let $X=\{aa,ab,ba\}$ and let $f$ be the coding morphism defined by 
$f(u)=aa$, $f(v)=ab$, $f(w)=ba$.  
The two partitions of $[0,1[$ corresponding to $T_f$ are
\begin{displaymath}
I_{u}=[0,1-2\alpha[,\quad I_{v}=[1-2\alpha,1-\alpha[\quad I_{w}=[1-\alpha,1[
\end{displaymath}
and
\begin{displaymath}
J_{v}=[0,\alpha[,\quad J_{w}=[\alpha,2\alpha[\quad J_{u}=[2\alpha,1[.
\end{displaymath}
The transformation $T_f$ is represented in Figure~\ref{figure3interval}.
\begin{figure}[hbt]
\centering\gasset{Nh=2,Nw=2,ExtNL=y,NLdist=2,AHnb=0,ELside=r}
\begin{picture}(100,15)
\node[fillcolor=red](0h)(0,10){$0$}
\node[fillcolor=blue](1-2alpha)(23.6,10){$1-2\alpha$}
\node[fillcolor=green](1-alpha)(61.8,10){$1-\alpha$}
\node(1h)(100,10){$1$}
\drawedge[linecolor=red,linewidth=1](0h,1-2alpha){$u$}
\drawedge[linecolor=blue,linewidth=1](1-2alpha,1-alpha){$v$}
\drawedge[linecolor=green,linewidth=1](1-alpha,1h){$w$}

\node[fillcolor=blue](0b)(0,0){$0$}
\node[fillcolor=green](alpha)(38.2,0){$\alpha$}
\node[fillcolor=red](2alpha)(76.4,0){$2\alpha$}\node(1b)(100,0){$1$}
\drawedge[linecolor=blue,linewidth=1](0b,alpha){$v$}
\drawedge[linecolor=green,linewidth=1](alpha,2alpha){$w$}
\drawedge[linecolor=red,linewidth=1](2alpha,1b){$u$}
\end{picture}
\caption{The transformation $T_f$.}\label{figure3interval}
\end{figure}
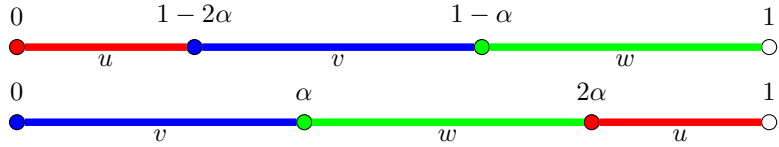
It is actually a representation on $3$ intervals of the rotation of angle
$2\alpha$.
Note that the point $z=1-\alpha$ is a separation point which is not
a singularity of $T_f$.
\begin{table}[hbt]
\begin{displaymath}
\begin{array}{|l|l|}\hline
(X,<_1) & (X,<_2)\\ \hline
aa,ab,ba&ab,ba,aa\\ \hline
a,baab,bab&bab,baab,a\\\hline
aa,aba,b&b,aba,aa\\ \hline
\end{array}
\end{displaymath}
\caption{The two orders on the three $S$-maximal
bifix codes of $S$-degree $2$.}\label{tableExchange2}
\end{table}
The first row of Table~\ref{tableExchange2} gives
the two orders on $X$. The next two rows give
the two orders for each of the two other $S$-maximal
bifix codes of $S$-degree $2$ (there are actually exactly three
$S$-maximal
bifix codes of $S$-degree $2$ in the Fibonacci set, see~\cite{BerstelDeFelicePerrinReutenauerRindone2012}).
\end{example}

Let $T$ be a minimal interval exchange transformation on the alphabet $A$.
Let $x$ be the natural coding of $T$
relative to some $z\in[0,1[$.
Set $S=F(x)$.
Let $X$ be a finite $S$-maximal bifix code. 
Let $f:B^*\rightarrow A^*$ be a morphism which maps bijectively
$B$ onto $X$. Since $S$ is recurrent, the set $X$
is an $S$-maximal prefix code. Thus $x$ has a prefix $x_0\in X$.
Set $x=x_0x'$. In the same way $x'$ has a prefix $x_1$
in $X$. Iterating this argument, we see that $x=x_0x_1\cdots$ with
$x_i\in X$. Consequently, there exists an
infinite word
$y$  on the alphabet $B$ such
that $x=f(y)$. The word $y$ is the \emph{decoding} of the
infinite
word $x$ with respect to $f$. 
\begin{proposition}\label{propositionDecoding}
The decoding of $x$ with respect to $f$ is the natural coding of
the transformation associated  with $f$ relative to $z$:
$\Sigma_T(z)=f(\Sigma_{T_f}(z))$. 
\end{proposition}
\begin{proof}
Let $y=b_0b_1\cdots$ be the decoding of $x$ with respect to $f$. 
Set $x_i=f(b_i)$ for $i\ge 0$. Then,
for any $n\ge 0$, we have
\begin{equation}
T_f^n(z)=T^{|u_n|}(z) \label{eqTXn}
\end{equation}
with $u_n=x_0\cdots x_{n-1}$ (note that $|u_n|$ denotes the length of
$u_n$ with respect to the alphabet $A$).
Indeed, this is is true for $n=0$. Next
$T_f^{n+1}(z)=T_f(t)$ with $t=T_f^n(z)$.
Arguing by induction, we have $t=T^{|u_n|}(z)$. 
Since $x=u_nx_nx_{n+1}\cdots$, 
$t$ is in $I_{x_n}$ by \eqref{eqIw}. Thus
by Proposition~\ref{propositionTransformation}, 
$T_f(t)=T^{|x_n|}(t)$ and we obtain
$T_f^{n+1}(z)=T^{|x_n|}(T^{|u_n|}(z))=T^{|u_{n+1}|}(z)$
proving \eqref{eqTXn}.
Finally, for $u=f(b)$ with $b\in B$,
\begin{displaymath}
b_n=b \Longleftrightarrow x_n=u\Longleftrightarrow
T^{|u_n|}(z)\in I_{u}\Longleftrightarrow
T_f^n(z)\in I_{u}=K_b
\end{displaymath}
showing that $y$ is the natural coding of $T_f$ relative to $z$.

\end{proof}
\begin{example}
 Let $T,\alpha,X$ and $f$ be as in Example~\ref{exampleBifixDegree2}.
Let $x=abaababa\cdots$ be the Fibonacci word. We have $x=\Sigma_T(\alpha)$.
 The decoding of $x$ with respect to $f$
is
$y=vuwwv\cdots$.
\end{example}

%%%%%%%%%%%%%%%%%%%
\subsection{Bifix codes and regular transformations}
The following result shows that, for the coding morphism
$f$ of a finite $S$-maximal bifix code, the map $T\mapsto T_f$ 
preserves the regularity of the transformation.

\begin{theorem}\label{theoremMinimal}
Let $T$ be a regular interval exchange transformation and let $S=F(T)$.
For any finite $S$-maximal bifix code $X$ with coding morphism $f$,
the transformation
$T_f$ is regular.
\end{theorem}
\begin{proof}
Set $A=\{a_1,a_2,\ldots,a_s\}$ with $a_1<_1 a_2<_1\cdots <_1a_s$.
We denote  $\delta_i=\delta_{a_i}$. 
By hypothesis,  the orbits of $\delta_2,\ldots,\delta_s$
are infinite and disjoint.
Set $X=\{x_1,x_2,\ldots,x_t\}$ with $x_1<_1 x_2<_1\cdots<_1 x_t$.
Let $d$ be the $S$-degree of $X$.

For $x\in X$, denote by $\delta_x$ the left boundary of the
semi-interval $J_x$.
For each $x\in X$, it follows from Equation~\eqref{eqJu}
that there is an $i\in\{1,\ldots,s\}$ such that
$\delta_x=T^k(\delta_i)$ with $0\le k< |x|$. 
Moreover, we have $i=1$ if and only if $x=x_1$.
Since $T$ is regular, the index $i\ne 1$ and the integer $k$ are unique
for each $x\ne x_1$.
And for such $x$ and $i$, by \eqref{eqJw}, we have $\Sigma_T(\delta_i)=u\Sigma_T(\delta_x)$
with $u$ a proper suffix of $x$.

We  now  show that the orbits of
$\delta_{x_2},\ldots,\delta_{x_t}$  for the transformation $T_f$
are infinite and disjoint. Assume that $\delta_{x_p}=T_f^n(\delta_{x_q})$ for
some $p,q\in\{2,\ldots,t\}$ and $n\in \Z$. Interchanging $p,q$
if necessary, we may assume that $n\ge 0$.
Let $i,j\in\{2,\ldots,s\}$ be such that $\delta_{x_p}=
T^k(\delta_i)$ with $0\le k<|x_p|$
and  $\delta_{x_q}=
T^\ell(\delta_j)$ with $0\le \ell<|x_q|$. Since
$T^k(\delta_i)=T_f^n(T^{\ell}(\delta_j))=T^{m+\ell}(\delta_j)$
for some $m\ge 0$, we cannot have $i\ne j$ since
otherwise
the orbits of $\delta_i,\delta_j$ for the transformation $T$
intersect.
Thus $i=j$. Since $\delta_{x_p}=T^k(\delta_i)$,
 we have $\Sigma_T(\delta_i)=u\Sigma_T(\delta_{x_p})$
with $|u|=k$, and $u$ a proper suffix of $x_p$. And since 
$\delta_{x_p}=T_f^n(\delta_{x_q})$, we have 
$\Sigma_T(\delta_{x_q})=x\Sigma_T(\delta_{x_p})$ with $x\in X^*$.
Since on the other hand $\delta_{x_q}=T^\ell(\delta_i)$,
we have $\Sigma_T(\delta_i)=v\Sigma_T(\delta_{x_q})$ with $|v|=\ell$
and $v$ a proper suffix of $x_q$. We obtain
\begin{eqnarray*}
\Sigma_T(\delta_i)&=&u\Sigma_T(\delta_{x_p})\\
&=&v\Sigma_T(\delta_{x_q})=vx\Sigma_T(\delta_{x_p}).
\end{eqnarray*}
Since $|u|=|vx|$, this implies $u=vx$.
But since $u$ cannot have a suffix in
$X$, $u=vx$ implies $x=1$ and thus
$n=0$ and $p=q$. This concludes the proof.
\end{proof}
Let $f$ be a coding morphism for
a finite $S$-maximal bifix code $X\subset S$.
The set $f^{-1}(S)$ is called a \emph{maximal bifix decoding} of $S$.

\begin{theorem}\label{corollaryInverseImage}
The family of regular interval exchange sets is closed under maximal
bifix decoding.
\end{theorem}
\begin{proof}
Let $T$ be a regular interval exchange transformation such that
$S=F(T)$.
By Theorem~\ref{theoremMinimal},  $T_f$ is a regular interval
exchange transformation. We show that $f^{-1}(S)=F(T_f)$, which
implies
the conclusion.

 Let
$x=\Sigma_T(z)$ for some $z\in [0,1[$ and let $y=f^{-1}(x)$. 
Then $S=F(x)$ and $F(T_f)=F(y)$. For any   $w\in F(y)$, we have $f(w)\in F(x)$
and thus $w\in f^{-1}(S)$. This shows that $F(T_f)\subset f^{-1}(S)$.
Conversely, let $w\in f^{-1}(S)$ and let $v=f(w)$. Since $S=F(x)$,
there
is a word $u$ such that $uv$ is a prefix of $x$. Set $z'=T^{|u|}(z)$
and $x'=\Sigma_T(z')$. Then $v$ is a prefix of $x'$ and $w$
is a prefix of $y'=f^{-1}(x')$. Since $T_f$ is regular, it is minimal
and thus $F(y')=F(T_f)$. This implies that $w\in F(T_f)$.
\end{proof}

Since a regular interval exchange set is uniformly recurrent,
Theorem~\ref{corollaryInverseImage} implies in particular
that if $S$ is a regular interval exchange set and $f$ a coding
morphism of a finite $S$-maximal bifix code, then $f^{-1}(S)$
is uniformly recurrent. This is not true for an arbitrary uniformly
recurrent set $S$, as shown by the following example.

\begin{example}
Set $A=\{a,b\}$ and $B=\{u,v\}$.
Let $S$ be the set of factors of $(ab)^*$ and let $f:B^*\rightarrow
A^*$
be defined 
by $f(u)=ab$ and $f(v)=ba$. Then $f^{-1}(S)=u^*\cup v^*$ which is not recurrent.
\end{example}

We illustrate the proof of Theorem~\ref{theoremMinimal}
 in the following example.
\begin{example}\label{exampleBifixmaxDegree3}
Let $T$ be the rotation of angle $\alpha=(3-\sqrt{5})/2$.
The set $S=F(T)$ is the Fibonacci set.
Let $X=\{a,baab,babaabaabab,babaabab\}$. The set $X$ is an $S$-maximal
bifix code of $S$-degree $3$ 
(see~\cite{BerstelDeFelicePerrinReutenauerRindone2012}). 
The values of the $\mu_{x_i}$ (which are the right boundaries
of the intervals $I_{x_i}$) and
$\delta_{x_i}$ are represented in Figure~\ref{exampleDegree3}.
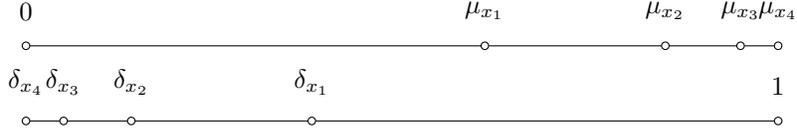
\begin{figure}[hbt]
\centering
\gasset{Nw=1,Nh=1,AHnb=0,ExtNL=y,NLdist=3}
\begin{picture}(100,20)
\node(0h)(0,10){$0$}\node(1h)(61,10){$\mu_{x_1}$}\node(2h)(85,10){$\mu_{x_2}$}\node(3h)(95,10){$\mu_{x_3}$}
\node(4h)(100,10){$\mu_{x_4}$}
\node(0b)(0,0){$\delta_{x_4}$}\node(1b)(5,0){$\delta_{x_3}$}\node(2b)(14,0){$\delta_{x_2}$}\node(3b)(38,0){$\delta_{x_1}$}\node(4b)(100,0){$1$}

\drawedge(0h,1h){}\drawedge(1h,2h){}\drawedge(2h,3h){}\drawedge(3h,4h){}
\drawedge(0b,1b){}\drawedge(1b,2b){}\drawedge(2b,3b){}\drawedge(3b,4b){}
\end{picture}
\caption{The transformation associated with a bifix code of $S$-degree
  $3$.}
\label{exampleDegree3}
\end{figure}

The infinite word $\Sigma_T(0)$ is represented in
Figure~\ref{exampleDegree3b}.
The value indicated on the word $\Sigma_T(0)$ after a prefix $u$ is 
$T^{|u|}(0)$. The three values
$\delta_{x_4},\delta_{x_2},\delta_{x_3}$ correspond to the three
prefixes
of $\Sigma_T(0)$ which are proper suffixes of $X$.
\begin{figure}[hbt]
\centering
\gasset{Nw=1,Nh=1,AHnb=0}
\begin{picture}(50,18)(0,-8)
\node[Nframe=n](-1)(-10,0){$\Sigma_T(0)$=}
\node[Nframe=n](0h)(0,8){$\delta_{x_4}$}
%\node(0b)(0,-8){}
\node[Nframe=n](0)(0,0){}
\node[Nframe=n](a1)(3,0){$a$}
\node[Nframe=n](1)(6,0){}\node[Nframe=n]((a2)(9,0){$a$}
\node[Nframe=n](2)(12,0){}\node[Nframe=n]((a3)(15,.3){$b$}
\node[Nframe=n](3h)(18,8){$\delta_{x_2}$}
\node[Nframe=n](3)(18,0){}
\node[Nframe=n]((a4)(21,0){$a$}
\node[Nframe=n](4)(24,0){}\node[Nframe=n]((a5)(27,0){$a$}
\node[Nframe=n](5)(30,0){}\node[Nframe=n]((a3)(33,.3){$b$}
\node[Nframe=n](6)(36,0){}\node[Nframe=n]((a4)(39,0){$a$}
\node[Nframe=n](7)(42,0){}\node[Nframe=n]((a5)(45,.3){$b$}
\node[Nframe=n](8h)(48,8){$\delta_{x_3}$}
\node[Nframe=n](8)(48,0){}
\node[Nframe=n]((a6)(51,0){$a$}
\node[Nframe=n](9)(54,0){}\node[Nframe=n](a7)(57,0){$\cdots$}
\drawedge[dash={0.2 0.5}0](0,0h){}
\drawedge[dash={0.2 0.5}0](3h,3){}
\drawedge[dash={0.2 0.5}0](8h,8){}
\end{picture}
\caption{The infinite word $\Sigma_T(0)$.}
\label{exampleDegree3b}
\end{figure}
\end{example}
The following example shows that Theorem~\ref{corollaryInverseImage}
is not true when $X$ is not bifix.
\begin{example}\label{exampleNeutralG}
Let $S$ be the Fibonacci set and let $X=\{aa,ab,b\}$. The set
$X$ is an $S$-maximal prefix code. Let $B=\{u,v,w\}$
and let $f$ be the coding morphism for $X$ defined by $f(u)=aa$,
$f(v)=ab$, $f(w)=b$. The set $W=f^{-1}(S)$ is not an interval exchange
set. Indeed, we have $vu,vv,wu,wv\in W$. This implies that both
$J_v$ and $J_w$ meet $I_u$ and $I_v$, which is impossible
in an interval exchange transformation.
\end{example}
%%%%%%%%%%%%%%%%%%%%%%%%%%%%%%%%%%%
\section{Tree sets}\label{sectionTreePlanarTree}
We  introduce in this section the notions of 
tree sets and planar tree sets. We first introduce the notion of extension graph which describes
the possible two-sided extensions of a word.
%%%%%%%%%%%%%%%%%%
\subsection{Extension graphs}

Let $S$ be a biextendable set of words.
 For $w\in S$,
we denote
\begin{displaymath}
L(w)=\{a\in A\mid aw\in S\},\quad
R(w)=\{a\in A\mid wa\in S\}
\end{displaymath}
and
\begin{displaymath}
E(w)=\{(a,b)\in A\times A\mid awb\in S\}.
\end{displaymath}
For  $w\in S$,  the \emph{extension graph}
of $w$ is the undirected bipartite graph $G(w)$ on the set of vertices
which is the disjoint union of two copies of
$L(w)$ and $R(w)$ with edges the pairs 
$(a,b)\in E(w)$.

Recall that an undirected graph is a tree if it is connected and acyclic.

Let $S$ be a  biextendable set.  We say that $S$ is a \emph{tree set}
if the graph $G(w)$ is a tree for all $w\in S$.

Let $<_1$ and $<_2$ be two orders on $A$. For a set $S$ and
a word $w\in S$, we say that
the
graph $G(w)$ is \emph{compatible} with the orders $<_1$ and $<_2$ if 
for any $(a,b),(c,d)\in E(w)$, one has
\begin{displaymath}
a<_1 c\Longrightarrow b\le_2 d.
\end{displaymath}
Thus, placing the vertices of $L(w)$ ordered by $<_1$ on a line and those of
$R(w)$ ordered by $<_2$ on a parallel line, the edges of
the graph may be drawn 
as straight noncrossing segments, resulting in a planar graph.

We say that a biextendable set $S$ is a \emph{planar tree set} 
with respect to two orders $<_1$ and $<_2$ on $A$ if
for any $w\in S$, the graph $G(w)$ is a tree compatible with $<_1,<_2$.
Obviously, a planar tree set is a tree set.

The following example shows that the Tribonacci set is not a planar tree set.
\begin{example}\label{exampleTribonacciPasPlanaire}
Let $S$ be the Tribonacci set (see Example~\ref{exampleTribonacci}).
The words $a,aba$ and $abacaba$ are bispecial. Thus
 the words $ba,caba$ are right-special and the
words $ab, abac$ are left-special.
The graphs $G(\varepsilon),G(a)$ and $G(aba)$ are shown in
Figure~\ref{figureTribonacci}.
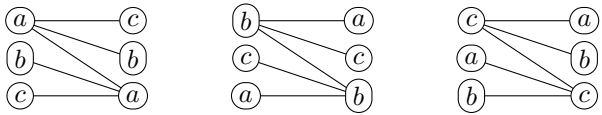
\begin{figure}[hbt]
\centering
\gasset{Nadjust=wh,AHnb=0}
\begin{picture}(60,10)
\put(0,0){
\begin{picture}(15,10)
\node(a)(0,10){$a$}\node(b)(0,5){$b$}\node(c)(0,0){$c$}
\node(c')(15,10){$c$}\node(b')(15,5){$b$}\node(a')(15,0){$a$}

\drawedge(a,a'){}\drawedge(a,b'){}\drawedge(a,c'){}
\drawedge(b,a'){}\drawedge(c,a'){}
\end{picture}
}
\put(30,0){
\begin{picture}(15,10)
\node(b)(0,10){$b$}\node(c)(0,5){$c$}\node(a)(0,0){$a$}
\node(a')(15,10){$a$}\node(c')(15,5){$c$}\node(b')(15,0){$b$}

\drawedge(b,a'){}\drawedge(b,b'){}\drawedge(b,c'){}
\drawedge(c,b'){}\drawedge(a,b'){}
\end{picture}
}
\put(60,0){
\begin{picture}(15,10)
\node(c)(0,10){$c$}\node(a)(0,5){$a$}\node(b)(0,0){$b$}
\node(a')(15,10){$a$}\node(b')(15,5){$b$}\node(c')(15,0){$c$}

\drawedge(c,a'){}\drawedge(c,b'){}\drawedge(c,c'){}
\drawedge(a,c'){}\drawedge(b,c'){}
\end{picture}
}
\end{picture}
\caption{The graphs $G(\varepsilon),G(a)$ and $G(aba)$ in the Tribonacci
  set.}
\label{figureTribonacci}
\end{figure}
One sees easily that it not possible to find two orders on $A$ making
the three graphs planar. 
\end{example}
%%%%%%%%%%%%%%%%%%%%%%%%%%%%%%
\subsection{Interval exchange sets and planar tree sets}\label{sectionIntervalPlanar}
The following result is proved in~\cite{FerencziZamboni2008}
with a converse (see below).
\begin{proposition}\label{propositionExchangeTreeCondition}
Let $T$ be an interval exchange transformation on $A$
ordered by $<_1$ and $<_2$.
If $T$ is regular, the set $F(T)$ is a planar tree set
with respect to $<_2$ and $<_1$.
\end{proposition}
\begin{proof}
Assume that $T$ is a regular interval exchange transformation relative
to $(I_a,\alpha_a)_{a\in A}$ and
let $S=F(T)$.

Since $T$ is minimal, $w$ is in $S$ if and only if $I_w\ne\emptyset$.
Thus, one has 
\begin{enumerate}
\item[(i)] $b\in R(w)$ if and only if $I_w\cap T^{-|w|}(I_b)\ne\emptyset$
and 
\item[(ii)]$a\in L(w)$ if and only if $J_a\cap
I_{w}\ne\emptyset$.
\end{enumerate}
Condition (i) holds because $I_{wb}=I_w\cap T^{-|w|}(I_b)$ and condition
(ii) because $I_{aw}=I_a\cap T^{-1}(I_w)$, which implies
$T(I_{aw})=J_a\cap I_w$.
In particular, (i) implies that $(I_{wb})_{b\in R(w)}$ is an ordered
partition
of $I_w$ with respect to $<_1$.

We say that a path in a graph is \emph{reduced} if it does not use
consecutively the same edge.
For $a,a'\in L(w)$ with $a<_2 a'$, there is a unique reduced
path in $G(w)$ from $a$ to
$a'$ which is the sequence
 $a_1,b_1,\ldots a_n$ with $a_1=a$ and $a_n=a'$ with
$a_1<_2a_2<_2\cdots<_2a_n$, $b_1<_1b_2<_1\cdots <_1 b_{n-1}$ and
$J_{a_{i}}\cap I_{wb_i}\ne\emptyset$, $J_{a_{i+1}}\cap
I_{wb_{i}}\ne\emptyset$
for $1\le i\le n-1$ (see Figure~\ref{figG(w)}). Note that
the hypothesis that $T$ is regular is needed here since otherwise
the right boundary of $J_{a_i}$ could be the left boundary of
$I_{wb_i}$. Thus $G(w)$ is a tree.
It is compatible with $<_2,<_1$ since the above shows that $a<_2 a'$ 
implies that the letters $b_1,b_{n-1}$ such that
$(a,b_1),(a',b_{n-1})\in E(w)$
satisfy $b_1\le_1 b_{n-1}$.
\end{proof}
\begin{figure}[hbt]
\centering
\gasset{Nadjust=wh,AHnb=0}
\begin{picture}(120,12)(-10,0)
\node(H1)(0,10){}
\node[Nframe=n](HB2)(5,10){}\node(H2)(20,10){}\node(H3)(30,10){}\node(H4)(50,10){}
\node(H5)(80,10){}\node(H6)(100,10){}

\node(B1)(-10,0){}\node(B2)(5,0){}\node(B3)(15,0){}
\node[Nframe=n](BH2)(20,0){}\node[Nframe=n](BH3)(30,0){}\node(B4)(35,0){}
\node(B5)(65,0){}\node[Nframe=n](BH5)(80,0){}\node(B6)(85,0){}\node(B7)(95,0){}
\node[Nframe=n](BH6)(100,0){}\node(B8)(110,0){}

\drawedge(H1,H2){$I_{wb_1}$}\drawedge(H3,H4){$I_{wb_2}$}
\drawedge(H5,H6){$I_{w_{b_{n-1}}}$}

\drawedge(B1,B2){$J_{a_1}$}\drawedge(B3,B4){$J_{a_2}$}\drawedge(B5,B6){$J_{a_{n-1}}$}\drawedge(B7,B8){$J_{a_n}$}

\drawedge[dash={0.2 0.5}0](B2,HB2){}\drawedge[dash={0.2 0.5}0](BH2,H2){}
\drawedge[dash={0.2 0.5}0](BH3,H3){}
\drawedge[dash={0.2 0.5}0](BH5,H5){}\drawedge[dash={0.2 0.5}0](BH6,H6){}
\end{picture}
\caption{A path from $a_1$ to $a_n$ in $G(w)$.}\label{figG(w)}
\end{figure}
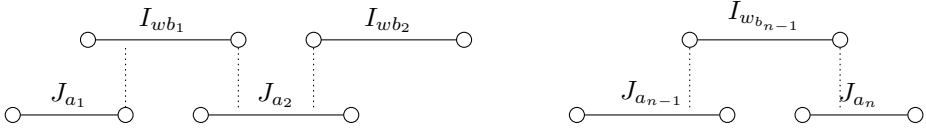

By Proposition~\ref{propositionExchangeTreeCondition}, a regular interval 
exchange set is a planar tree set, and thus in particular a tree set.
Note that the analogue of Theorem~\ref{corollaryInverseImage} holds
for the class of uniformly recurrent tree sets~\cite{BertheDeFeliceDolceLeroyPerrinReutenauerRindone2013c}.

The main result of~\cite{FerencziZamboni2008} states that a 
uniformly recurrent set
$S$ on an alphabet $A$ is a regular interval exchange set if and only
if
$A\subset S$ and
there exist two orders $<_1$ and $<_2$ on $A$ such that the
following conditions are satisfied for any word $w\in S$.
\begin{enumerate}
\item[(i)] The set $L(w)$ (resp. $R(w)$)
is formed of consecutive elements for the order $<_2$ (resp. $<_1$).
\item[(ii)] For $(a,b),(c,d)\in E(w)$, if $a<_2 c$, then $b\le_ 1 d$.
\item [(iii)] If $a,b\in L(w)$ are consecutive for the order $<_2$, then
the set $R(aw)\cap R(bw)$ is a singleton.
\end{enumerate}
It is easy to see that a biextendable set $S$ containing $A$ satisfies (ii) and (iii)
if and only if it is a planar tree set. Actually, in this
case, it
automatically
satisfies also condition (i). Indeed, let us consider a word $w$
and $a,b,c\in A$ with $a<_1 b<_1 c$ such that $wa,wc\in S$ but
$wb\notin S$. Since $b\in S$ there is a (possibly empty) suffix $v$ of $w$ such that
$vb\in S$. We choose $v$ of maximal length. Since $wb\notin S$, we
have $w=uv$ with $u$ nonempty. Let $d$ be the last letter of $u$.
Then we have $dva,dvc\in S$ and $dvb\notin S$. Since $G(v)$
is a tree and $b\in R(v)$, there is a letter $e\in L(v)$ such
that $evb\in S$. But $e<_2d$ and $d<_2 e$ are both impossible
since $G(v)$ is compatible with $<_2$ and $<_1$. Thus we reach
a contradiction.

This shows that the following reformulation of the main result
of~\cite{FerencziZamboni2008} is equivalent to the original one.
\begin{theorem}[Ferenczi, Zamboni]\label{theoremFZ}
A set $S$ is a regular interval exchange set on the alphabet $A$ if and only if
it is a uniformly recurrent planar tree set containing $A$.
\end{theorem}

We have already seen that the Tribonacci set is a tree set
which is not a planar tree set (Example~\ref{exampleTribonacciPasPlanaire}).
The next example shows that there are uniformly recurrent tree sets
which are neither Sturmian nor regular interval exchange sets.
\begin{example}
Let $S$ be the Tribonacci set on the alphabet $A=\{a,b,c\}$
and let $f:\{x,y,z,t,u\}^*\rightarrow A^*$ be the coding
morphism for $X=S\cap A^2$ defined by
$f(x)=aa$, $f(y)=ab$, $f(z)=ac$, $f(t)=ba$, $f(u)=ca$.
By Theorem 7.1
in~\cite{BertheDeFeliceDolceLeroyPerrinReutenauerRindone2013c},
the set $W=f^{-1}(S)$ is a uniformly recurrent tree set.
It is not Sturmian since $y$ and $t$ are two right-special words.
It is not either a regular
interval exchange set. Indeed, for any right-special word $w$ of $W$,
one has $\Card(R(w))=3$. This is not possible in a regular interval
exchange set $T$ since, $\Sigma_T$ being injective, the length of
the interval $J_w$ tends to $0$ as $|w|$ tends to infinity
and it cannot contain several separation points.
It can of course also be verified directly that $W$ is not a planar tree set.
\end{example}
%%%%%%%%%%%%%%%%%%%%%%%%%%%%%%%%%
\subsection{Exchange of pieces}\label{sectionExchangePieces}
In this section, we show how one can define a generalization
of interval exchange transformations called exchange of
pieces. In the same way as interval exchange is a generalization
of rotations on the circle, exchange of pieces is a generalization
of rotations of the torus. We begin by studying this direction
starting from the Tribonacci word. For more on the Tribonacci
word, see~\cite{Rauzy1979} and also~\cite[Chap. 10]{Lothaire2005}.

\paragraph{The Tribonacci shift}
The Tribonacci set $S$ is not an interval exchange set  but it is however   
the  natural coding of  another type of  geometric 
transformation, namely an exchange of pieces in the plane, 
 which  is also     a translation acting  on the two-dimensional torus 
${\mathbb T}^2$. This will allow us to show   that 
 the   decoding of the Tribonacci word with respect to    a coding 
morphism  for   a  finite $S$-maximal  bifix code is again  
a  natural coding of an exchange of pieces.
 
 The \emph{Tribonacci shift} 
is the symbolic dynamical system $(M_x,\sigma)$, where 
$M_x = \overline{\{\sigma^n(x) :\, n \in \mathbb{N}\}}$ is the closure of the $\sigma$-orbit of~$x$ where $x$ is the Tribonacci word.
By uniform recurrence of the Tribonacci word,  $(M_x,\sigma)$  is minimal 
and $M_{x} = M_{y}$ for each $y \in M_x$ 
(\cite[Proposition~4.7]{Queffelec2010}). 
The Tribonacci  set is the set of factors of the 
Tribonacci  shift $(M_x, \sigma)$.

 \paragraph{Natural coding}
Let $\Lambda$ be a
full-rank lattice in $\mathbb{R}^d$. We say that an infinite word
$x$ is a {\em natural coding} of a  toral translation $T_\mathbf{t}: {\mathbb R}^d/\Lambda   \rightarrow {\mathbb R}^d/\Lambda, \  \mathbf{x} \mapsto \mathbf{x} + \mathbf{t}$  if there exists a fundamental domain $R$ for $\Lambda$ together with a partition $R=R_1\cup\cdots\cup R_k$
such that on each $R_i$ ($1\le i\le k$), there exists a vector ${\mathbf t}_i$ such that  the map $T_\mathbf{t}$ is   given by  the translation along ${\mathbf t}_i$, and $x$ is the coding of a point $\mathbf{x}\in R$ with respect to this partition. A symbolic dynamical system $(M,\sigma)$ is a 
{\em natural coding} of $({\mathbb R}^d/\Lambda,T_\mathbf{t} )$ 
 if  every element of $M$ is a natural  coding of the orbit of some point of the $d$-dimensional torus ${\mathbb R}^d/\Lambda$
(with respect to the same  partition) 
and if, furthermore, $(M,\sigma)$ and $({\mathbb R}^d/\Lambda,T_\mathbf{t} )$ 
are measurably conjugate.

\paragraph{Definition of the Rauzy fractal} Let $\beta$ stand for the  
Perron-Frobenius  eigenvalue of the Tribonacci substitution.
It is the largest root of $z^3-z^2-z-1$. Consider the translation 
$R_{\beta}\colon {\mathbb T}^2 \rightarrow {\mathbb T }^2,$
$x \mapsto x+(1/\beta,1/\beta ^2)$.
Rauzy introduces  in  \cite{Rauzy1982} a fundamental domain for a  
two-dimensional lattice, called  the 
Rauzy fractal (it has indeed  fractal boundary), which  provides     a partition for 
  the symbolic dynamical system $(M_x,\sigma)$   to  be a natural
coding   for  $R_{\beta}$.  The Tribonacci word  is a natural coding  of  the orbit of 
the point $0$ under the action of the toral translation  in
${\mathbb T}^2$: $x \mapsto x +(\frac{1}{\beta}, \frac{1}{\beta^2})$.
Similarly as in the case of interval exchanges, we have  the following  commutative diagram
 \begin{displaymath}
\begin{array}{ccc}
\GT^2 &\stackrel{R_\beta}
{\longrightarrow}&\GT^2
\\
\Big\downarrow&&\Big\downarrow\\\
M_x&\stackrel{\sigma}{\longrightarrow}&M_x
\end{array}
\end{displaymath}

The {\em Abelianization map}  ${\bf f}$ of the
  free monoid  $\{1,2,3\}^*$  is  defined by
  ${\bf f}:\{1,2,3\}^{*}\rightarrow \Z^{3}, \
{\bf f}(w)=|w|_{1} { e}_{1}+ |w|_{2} { e}_{2}+|w|_{3} { 
e}_{3},$
  where  $|w|_i$ denotes the number of occurrences of the
  letter $i$ in the word $w$, and  $({ e}_{1},{ e}_{2},{ e}_{3})$ 
stands for the canonical
  basis of ${\mathbb R}^{3}$.

Let $f$ be the morphism $a\mapsto ab, b\mapsto ac,c\mapsto a$
such that the Tribonacci word is the fixpoint of $f$
(see Example~\ref{exampleTribonacci}).
 The   incidence  matrix $F$ of $f$  is defined   by
$ F = \left(|f(j)|_i \right)_{(i,j) \in {\mathcal A}^2},
$
where $|f(j)|_i$  counts the number of  occurrences of the letter  $i$ in  $f(j)$.
One has 
$  F=    \left[
\begin{array}{lll}
1&1&1\\
1&0&0\\
0&1&0
\end{array}
\right].$
% There is a  second action   defined on the Rauzy fractal defined as an exchange of pieces.
The incidence matrix $F$ admits 
as eigenspaces in ${\mathbb R}^3$
one  {\em expanding eigenline}  (generated by the eigenvector with positive coordinates
${ v}_{\beta}=(1/\beta, 1/\beta^2,1/\beta ^3)$ associated with the eigenvalue $\beta$).
% and 
%a {\em contracting eigenplane}. 
 We consider the  projection $\pi$ onto the  antidiagonal plane $ x+y+z=0 $  along the expanding direction of the matrix $F$.

One   associates with the Tribonacci
  word $x=(x_{n})_{n \geq 0}$ a broken line starting from $0$ in
  $\Z^{3}$ and approximating the expanding line
${ v}_{\beta}$
 as follows.  The {\em Tribonacci broken line}
is defined as the broken line which joins
with segments of length  $1$
 the  points
${\bf f}(x_0x_1\cdots x_{n-1}), \ n \in {\mathbb N}$.
In other words we describe this broken line by starting from
the origin, and then  by reading successively
the letters of the Tribonacci word
$x$,  going one step in direction
$e_i$ if one reads  the letter $i$.
The vectors ${\bf f}(x_0x_1\cdots x_n)$, $n \in {\mathbb N}$,
 stay within bounded distance of the
expanding
  line (this comes from the fact that $\beta$ is a Pisot number).
  The closure of the set of projected  vertices of the broken line is called the {\em Rauzy fractal} and is
 represented on Figure \ref{fig:RF}.
We    thus define the Rauzy fractal  ${\mathcal R}$ as
  \begin{displaymath}
{\mathcal R}:=
  \overline {\{ \pi ({\bf f}(x_{0}\cdots x_{n-1})); \ n \in \N \}},
\end{displaymath}
where $x_{0}\dots x_{n-1}$ stands  for the empty word when $n=0$.

\begin{figure}[h]
\begin{center}
  \includegraphics[width=.45\textwidth]{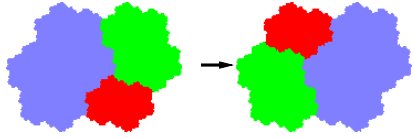}
\caption{The Rauzy fractal } \label{fig:RF}
\end{center}
\end{figure}

The  Rauzy fractal is divided into three pieces, for $i= \{1,2,3\}$
 \begin{eqnarray*}  
{\mathcal R}(i)&:=&
  \overline {\{ \pi ({\bf f}(x_{0}\cdots x_{n-1})); \  x_{n}=i, n \in \N \}},\\
   {\mathcal R}'(i)&:=&
  \overline {\{ \pi ({\bf f}(x_{0}\cdots x_{n})); \  x_{n}=i, n \in \N \}}.
\end{eqnarray*}
 It has been proved in \cite{Rauzy1982} that  these pieces have  non-empty interior  and are   disjoint up to a set of zero measure.
  The following exchange of pieces
 $E$ is thus  well-defined 
\begin{displaymath}
E:\mbox{Int }{\mathcal R}_1\cup
\mbox{Int }{\mathcal R}_2
\cup
\mbox{Int }{\mathcal R}_3
\rightarrow 
{\mathcal R} ,\ 
x \mapsto x+\pi({e}_i), \mbox{ when } x \in\mbox{Int }{\mathcal R}_i.
\end{displaymath}
One has  $E({\mathcal R}_i)={\mathcal R}'_i$, for all $i$.

We consider  the lattice $\Lambda$  generated  by the  vectors  $\pi({e}_i)- \pi({e}_j)$, for $i \neq j$.
The   Rauzy fractal tiles periodically the plane, that is,  $\cup_{ \gamma \in \Lambda} \gamma + {\mathcal R}$
is equal to the   plane $x+y+z=0$, and
for $\gamma \neq \gamma ' \in \Lambda$,  $\gamma + {\mathcal R}$   and $\gamma' + {\mathcal R}$ do not intersect (except on a set of zero measure).
This is why  the exchange of pieces is  in fact  measurably conjugate to the translation $R_{\beta}$.
Indeed the vector
of coordinates of 
$\pi({\bf f}(x_0x_1\cdots x_{n-1}))$ in the basis
  $(\pi(e_3)-\pi(e_1),\pi(e_3)-\pi(e_2))$
of the plane $x+y+z=0$   is  $ n \cdot(1/\beta, 1/\beta ^2) -
(|x_0x_1 \cdots x_{n-1}|_1,|x_0x_1 \cdots x_{n-1}|_2).$
Hence   the coordinates
of $E^n (0) $
in the  basis
$(_pi(e_3)-\pi(e_1),_pi(e_3)-_pi(e_2))$
are
equal  to
$R_{\beta}^n (0 )$ modulo ${\mathbb Z}^2$.

 \paragraph{Bifix codes and exchange of pieces}
Let
$(\RR_a)_{a\in A}$ and $(\RR'_a)_{a\in A}$ be two families
of subsets of a compact set $\RR$ incuded in $\R^d$. We assume that the 
families $(\RR_a)_{a\in A}$ and  $(\RR'_a)_{a\in A}$ both form a partition of $\RR$ up to
a set of zero measure. We assume that there exist vectors $e_a$
such that $\RR'_a=\RR_a+e_a$ for any $a\in A$.
The exchange of pieces associated with these
data is the map $E$ defined on $\RR$ (except a set of measure zero)
by $E(z)=z+e_a$ if $z\in\RR_a$. The notion of natural coding of
an exchange of pieces extends here in a natural way.

Assume that $E$ is an exchange of pieces as defined above.
Let $S$ be the set of factors of the natural codings of $E$.
We assume that $S$ is uniformly recurrent.

 %For a word $w \in \{1,2,3\}^{*}$, the notation $[w]$ stands for the cylinder  made of the infinite  words in $X_f$  that start with the prefix  $w$, that is, 
%  $[w] = \{y \in X_f  \mid y_0 \ldots y_{n-1} = w\}$. 
  By analogy with the case of interval exchanges, let 
  $I_{a}={\mathcal R}_a$ and let $J_a=E ({\mathcal R}_a)$.
   For a word $w \in A^{*}$, one defines  similarly as for interval exchanges $I_w$ and $J_w$.
 
Let $X$ be a finite $S$-maximal   prefix code. The family $I_w$, $w \in X$,  is a partition (up to sets of zero measure) of  ${\mathcal R}$.
 If $X$ is a finite $S$-maximal   suffix code, then the family $J_w$ is a partition (up to sets of zero measure) of  ${\mathcal R}$.
 Let $f$  be a coding morphism for $X$.
 If $X$ is a  finite $S$-maximal   bifix code, then $E_X$  is the exchange of pieces $E_f$ (defined as for interval exchanges), hence the decoding of $x$  with respect to $f$
 is the natural coding  of the exchange of pieces  associated with $f$.
In particular, $S$ being the Tribonacci set, the decoding of $S$
 by a finite $S$-maximal bifix code is again the natural coding of
an exchange of pieces.
 If $X$ is the set of factors of  length $n$ of $S$, then 
 $E_f$ is in fact  equal to $R^n _{\beta}$ (otherwise, there is no reason for  this exchange of pieces   to be a  translation).
 The analogues of Proposition~\ref{propositionTransformation}  and \ref{propositionDecoding}  thus hold here also.

%%%%%%%%%%%%%%%%%%%%%%%%%%%%%%%%%%%%%%%%
\subsection{Subgroups of finite index}\label{sectionSkew}
We denote by $F_A$ the free group on the set $A$.

Let $S$ be a recurrent set containing the alphabet $A$. We say that
$S$ has the \emph{finite index basis property} if the following holds:
a finite bifix code $X\subset S$ is an $S$-maximal bifix code
of $S$-degree $d$
if and only if it is a basis of a subgroup of index $d$ of $F_A$.

The following is a 
consequence of the main result of \cite{BertheDeFeliceDolceLeroyPerrinReutenauerRindone2014}. 

\begin{theorem}\label{theoremBasis}
A regular interval exchange set
 has the finite index basis property.
\end{theorem}
\begin{proof}
Let $T$ be a regular interval exchange transformation and let $S=F(T)$.
Since $T$ is regular, $S$ is uniformly recurrent and
by Proposition~\ref{propositionExchangeTreeCondition}, it
 is a  tree set.
By Theorem 4.4 in~\cite{BertheDeFeliceDolceLeroyPerrinReutenauerRindone2014},
 a uniformly recurrent tree set has the finite index basis property,
and thus the conclusion follows.
\end{proof}
Note that Theorem~\ref{theoremBasis} implies in particular that
if $T$ is a regular $s$-interval exchange set and if $X$ is
a finite $S$-maximal bifix code of $S$-degree $d$, then
$\Card(X)=d(s-1)+1$. Indeed, by Schreier's Formula a basis
of a subgroup of index $d$ in a free group of rank $s$ has
$d(s-1)+1$ elements.

We use Theorem~\ref{theoremBasis} to give another proof of Theorem
\ref{theoremMinimal}. For this, we recall the following notion.

Let $T$ be an interval exchange transformation on $I=[0,1[$
relative to $(I_a)_{a\in A}$.
Let $G$ be a transitive permutation group on a finite set $Q$.
Let $\varphi:A^*\rightarrow G$ be a morphism and let $\psi$ be
the map from $I$ into $G$ defined by $\psi(z)=\varphi(a)$
if $z\in I_a$.
The \emph{skew product} of $T$ and $G$ is the transformation $U$
on $I\times Q$ defined by
\begin{displaymath}
U(z,q)=(T(z),q\psi(z))
\end{displaymath}
(where $q\psi(z)$ is the result of the action of the permutation
$\psi(z)$ on $q\in Q$).
Such a transformation is equivalent to an interval exchange transformation 
via the identification of $I\times Q$ with an interval obtained
by placing the $d=\Card(Q)$ copies of $I$ in sequence. This is called an
\emph{interval exchange transformation on a stack} in~\cite{BoshernitzanCarroll1997} (see also~\cite{Veech1975}). If $T$ is regular, then $U$ is also regular.

Let $T$ be a regular interval exchange transformation and
let $S=F(T)$. Let $X$ be a finite $S$-maximal bifix code of $S$-degree
$d=d_X(S)$.
By Theorem~\ref{theoremBasis}, $X$ is a basis of a subgroup $H$
of index $d$ of $F_A$. Let $G$ be the representation
of $F_A$ on the right cosets of $H$ and let $\varphi$
be the natural morphism from $F_A$ onto $G$. We identify
the right cosets of $H$ with the set $Q=\{1,2,\ldots,d\}$
with $1$ identified to $H$. Thus $G$ is a transitive permutation
group on $Q$ and $H$ is the inverse image by $\varphi$
of the permutations fixing $1$.

The transformation induced by the skew product
$U$ on $I\times\{1\}$ is clearly equivalent
to the transformation $T_f=T_X$ where $f$ is a coding morphism for the
$S$-maximal bifix code $X$. Thus $T_X$ is a regular interval exchange
transformation.

\begin{example}
Let $T$ be the rotation of Example~\ref{exampleFibonacciAlpha}. Let
$Q=\{1,2,3\}$ and let $\varphi$ be the morphism from $A^*$ into
the symmetric group on $Q$ defined by $\varphi(a)=(23)$ and $\varphi(b)=(12)$.
The transformation induced by the skew product of $T$ and $G$
on $I\times\{1\}$
corresponds to the bifix code $X$ of Example~\ref{exampleBifixmaxDegree3}.
For example, we have $U:(1-\alpha,1)\rightarrow (0,2)\rightarrow (\alpha,3)
\rightarrow (2\alpha,2)\rightarrow (3\alpha-1,1)$ 
(see Figure~\ref{figureSkew}) and the
corresponding word of $X$ is $baab$.
\begin{figure}[hbt]
\centering\gasset{Nh=2,Nw=2,ExtNL=y,NLdist=2,AHnb=0,ELside=r}
\begin{picture}(100,35)
\node[fillgray=0.1](0H)(0,30){$(0,3)$}
\node[Nadjust=n,fillgray=0.1](alpha)(38.2,30){$(\alpha,3)$}
\node[fillgray=.6](1-alphaH)(61.8,30){$(1-\alpha,3)$}
\node(1H)(100,30){$(1,3)$}
\drawedge[linegray=0.1,linewidth=1](0H,1-alphaH){$a$}
\drawedge[linegray=0.6,linewidth=1](1-alphaH,1H){$b$}

\node[fillgray=0.1](0h)(0,15){$(0,2)$}
\node[fillgray=0.6](1-alphah)(61.8,15){$(1-\alpha,2)$}
\node[Nadjust=n,fillgray=0.6](2alpha)(76.4,15){$(2\alpha,2)$}
\node(1h)(100,15){$(1,2)$}
\drawedge[linegray=0.1,linewidth=1](0h,1-alphah){$a$}
\drawedge[linegray=0.6,linewidth=1](1-alphah,1h){$b$}

\node[fillgray=0.1](0b)(0,0){$(0,1)$}
\node[Nadjust=n,fillgray=0.1](3alpha-1)(14.1,0){$(3\alpha-1,1)$}
\node[fillgray=0.6](1-alpha)(61.8,0){$(1-\alpha,1)$}
\node(1b)(100,0){$(1,1)$}
\drawedge[linegray=0.1,linewidth=1](0b,1-alpha){$a$}
\drawedge[linegray=0.6,linewidth=1](1-alpha,1b){$b$}

%%%%%%%%%%
\gasset{AHnb=1}
\drawedge(1-alpha,0h){}
\drawedge(0h,alpha){}
\drawedge(alpha,2alpha){}
\drawedge(2alpha,3alpha-1){}
\end{picture}
\caption{The transformation $U$.}\label{figureSkew}
\end{figure}
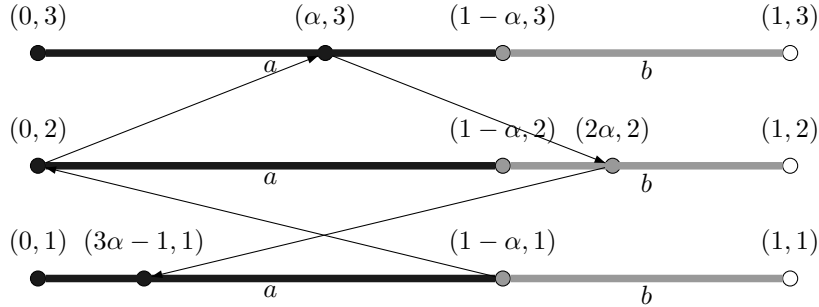
\end{example}

\bibliographystyle{plain}
\bibliography{bifixCodesIntervalExchange}

\end{document}